\documentclass[12pt]{amsart}
\usepackage{graphicx}
\usepackage{amsmath}
\usepackage{amsfonts}
\usepackage{amssymb}
\usepackage{mathtools}
\usepackage{amscd}
\usepackage[all,cmtip]{xy}
\usepackage{amsthm}
\usepackage{mathabx}
\usepackage{geometry}
\usepackage[capitalise]{cleveref}
\usepackage{xcolor}
\usepackage{comment}

\newtheorem{theorem}{Theorem}[section]
\newtheorem{proposition}[theorem]{Proposition}
\newtheorem{corollary}[theorem]{Corollary}
\newtheorem{lemma}[theorem]{Lemma}

\newtheorem{question}[theorem]{Question}

  \theoremstyle{definition}
\newtheorem{definition}[theorem]{Definition}
\newtheorem{example}[theorem]{Example}
\newtheorem{remark}[theorem]{Remark}

\newtheorem{observation}[theorem]{Observation}

\newcommand{\Z}{\mathbb{Z}}
\newcommand{\Q}{\mathbb{Q}}
\newcommand{\R}{\mathbb{R}}

\newcommand{\calA}{{\mathcal A}}

\newcommand{\calO}{{\mathcal O}}

\newcommand{\nbeq}{\begin{equation}}
\newcommand{\neeq}{\end{equation}}
\newcommand{\beq}{\begin{equation*}}
\newcommand{\eeq}{\end{equation*}}

\newcommand{\com}{_{\mathsf{com}}}

\DeclareMathOperator{\Min}{Min}
\DeclareMathOperator{\SL}{SL}
\DeclareMathOperator{\GL}{GL}
\DeclareMathOperator{\PGL}{PGL}

\DeclareMathOperator{\Hom}{Hom}

\DeclareMathOperator{\tr}{tr}
\DeclareMathOperator{\diag}{diag}

\DeclareMathOperator*{\hocolim}{hocolim}

\usepackage{mathabx,epsfig}

\begin{document}

\title{Centralizers and Classifying Spaces for Commutativity of $\SL_2(\Z[1/p])$}

\author{Omar Antol\'{\i}n-Camarena}
\address{Omar Antol\'{\i}n-Camarena, Universidad Nacional Auton\'onoma de M\'exico}
\email{omar@matem.unam.mx}
\author{Ramón Flores}
\address{Ramón Flores, Universidad de Sevilla}
\email{ramonjflores@us.es}
\author{Luis Jorge S\'anchez Salda\~na}
\address{Luis Jorge S\'anchez Salda\~na, Universidad Nacional Aut\'onoma de M\'exico}
\email{luisjorge@ciencias.unam.mx}
\begin{abstract}
    We compute the homotopy type of the classifying space for commutativity introduced by Adem, F. Cohen and Torres-Giese for the group $\SL_2(\Z[1/p])$, both as a
    discrete topological group and with the subspace topology from $\SL_2(\R)$. Doing so requires compiling detailed information on the maximal abelian subgroups of $\SL_2(\Z[1/p])$, which turn out to be precisely the centralizers of the non-central elements of the group, for which we employ methods from geometric group theory.
\end{abstract}

\maketitle

\section{Introduction}\label{sec:introduction}

Let $G$ be a topological group. For every $n\geq 0$, consider the subspace of $G^n$ given by
\[
    \operatorname{Hom}(\Z^n,G)
    \cong \{(g_1,\ldots,g_n)\in G^n\mid g_i g_j=g_j g_i\text{ for all }i,j\}.
\]
These spaces assemble into a simplicial subspace of the usual bar construction for $BG$: the face maps multiply adjacent entries or delete the first or last entry, and the degeneracy maps insert the identity element. The geometric realization of this simplicial space is called the \emph{classifying space for commutativity} \cite{AdemCohenTorresGiese} of $G$ and is denoted by $B\com G$.

The inclusion of commuting tuples into all tuples induces a a natural map from $B\com G$ to $BG$. The homotopy type of the space $E\com G$ can be described as the homotopy fiber of this map, but there is also a more precise definition of $E\com G$ up to homeomorphism, which we will recall in Section \ref{sec:prelim-com}. There is a notion of \emph{transitionally commutative principal $G$-bundle} classified by $B\com G$, and $E\com G$ is the total space of the universal such bundle.

So far most of the work on the classifying space for commutativity and the corresponding total space of the universal bundle has been carried out for Lie groups and discrete groups, and indeed most often the group is even assumed to be a compact and connected Lie group. Our goal in this paper is to test the waters outside of those extremely well-behaved topological groups and so we compute the homotopy type of $B\com G$ and $E\com G$ for the group $G = \SL_2(\Z[1/p])$, which is topologically less well-behaved, being a dense subgroup of $\SL_2(\R)$.

More precisely, if $G^\delta$ denotes $G$ equipped with the discrete topology, we first compute $B\com G^\delta$ and $E\com G^\delta$ up to homotopy equivalence, and then discuss what one obtains for $B\com G$ and $E\com G$ where $G$ carries its usual topology. We argue that the definition of $B\com G$ should probably be modified to use the fat geometric realization for groups with not-so-nice topologies, and show that for totally disconnected groups like $\SL_2(\Z[1/p])$ the modified definition will produce spaces that are weakly equivalent to $B\com G^\delta$ and $E\com G^\delta$. We show that for $G = \SL_2(\Z[1/p])$ the fat version of $B\com G$ is merely weakly equivalent to $B\com G^\delta$, not also homotopy equivalent to it; and we speculate on what would be obtained from the unmodified definition using geometric realization.

We now outline the contents of the paper. As mentioned above, Section \ref{sec:prelim-com} has some background information on classifying spaces for commutativity and on total spaces of the corresponding universal bundles. In particular it also cites the main results about them that we will use in our calculation.

Computing the classifying space for commutativity of a group requires detailed knowledge of its abelian subgroups and how they intersect. We provide the first ingredient for this in Section \ref{sec:ecom-sl2-domain}, where we show that for an integral domain $R$, the centralizers of non-central elements in $\SL_2(R)$ are abelian. Consequently, the maximal abelian subgroups of $\SL_2(R)$ are precisely the centralizers of non-central elements, and any two distinct maximal abelian subgroups intersect in the center. Using this information we prove that for any infinite integral domain $R$, we have:
\[
    E\com\SL_2(R)^\delta \simeq \bigvee_{\mathbb{N}} S^1.
\]
and similarly, that $B\com\SL_2(R)^\delta$ is a $K(\Gamma_R,1)$ for a group $\Gamma_R$ depending on the ring $R$, specifically, $\Gamma_R$ is the amalgamated product of all centralizers of non-central elements, the product being amalgamated over the center of the group.

Of course, that general description still leaves open the question of what these centralizers are exactly and the answer does depend on the specific ring $R$. We carry out the classification of the centralizers in the case of  $R = \Z[1/p]$ in Section \ref{sec:centralizers}. These centralizers are organized according to the type of the element: unipotent, split over $\Q$, real irreducible, or imaginary irreducible. In the irreducible cases the rank depends on the splitting behavior of $p$ in the associated quadratic field generated by an eigenvalue. These types are described in Section \ref{sec:prelim-sl2-localized}, prior to the aforementioned centralizer computation, and the consequences for $E\com$ and $B\com$ are collected in the short Section \ref{sec:consequences-ecom-bcom}. Namely, as already stated, $E\com \SL_2(\Z[1/p])^\delta$ (the $^\delta$ indicates we use the discrete topology) has the homotopy type of a countably infinite wedge of circles, and $B\com \SL_2(\Z[1/p])^\delta$ is a $K(\Gamma_p,1)$, where the group $\Gamma_p$ is is the amalgamated product of countably many copies of each of the following groups, amalgated over the obvious $C_2$ subgroup coming form the first factors: $C_2\times \Z[1/p]$, $C_2\times \Z$, $C_2\times \Z^2$, $C_2$, $C_4$, $C_6$, $C_4 \times \Z$, $C_6 \times \Z$.

Finally, in Section \ref{sec:topologies} we discuss the issue of the topology on $\Z[1/p]$, and how the classifying space for commutativity differs when using the usual topology versus the discrete topology. We end with a call for further investigation into classifying spaces for commutativity of groups with complicated topologies.

\subsection*{Acknowledgements}
The second author wish to thank the Department of Mathematics of the UNAM for their hospitality. He was supported by the program ``Becas de Movilidad" from the AUIP, and the project PID2024-155800NC32, funded by
the Spanish Ministry of Science, Innovation and Universities, and by FEDER, EU. The third author is grateful for the financial support of the DGAPA-UNAM grant PAPIIT~IN102426.

\section{Classifying spaces for commutativity}\label{sec:prelim-com}

As we said in the introduction, the classifying space for commutativity in a topological group is defined as the geometric realization of the simplicial space whose $n$-simplices are the $n$-tuples of pairwise commuting elements of $G$, that is, \[B\com G = |[n] \mapsto \Hom(\Z^n,G)|.\] The maps $\Hom(\Z^n, G) \to G^n$ recording the values of a homomorphism on the standard basis form a simplicial map and they geometrically realize to a map $B\com G \to BG$, whose homotopy fiber is $E\com G$.

To define $E\com$ up to homeomorphism, again one uses a simplicial model: it is the geometric realization of the simplicial space whose $n$ simplices consist of all $(n+1)$-tuples of elements of $G$ that \emph{commute affinely}, that is, tuples $(g_0, \ldots, g_n)$ such that the elements $\{g_i^{-1} g_j : 0 \le i, j\le n\}$ commute pairwise; or, equivalently, that all $g_j$ are contained in a single coset of some abelian subgroup. The face maps simply delete entries and the degeneracy maps insert identity elements, as in one of the two usual simplicial models for $EG$. The map $E\com G \to B\com G$ exhibiting $E\com G$ as the homotopy fiber of the map to $BG$ is given at the level of simplices by $(g_0, \ldots, g_n) \mapsto (g_0^{-1} g_1, g_1^{-1} g_2, \ldots, g_{n-1}^{-1} g_n)$. \cite{AdemGomez2015,AdemCohen2007}.

If $G$ is abelian, then every tuple is commuting, so $B\com G=BG$ and $E\com G$ agrees with the usual contractible total space $EG$. When $G$ is a discrete group a strong converse holds: if $E\com G$ is simply connected then $G$ is abelian (and thus $E\com G$ is even contractible). When $G$ is a compact, but not necessarily connected Lie group, there is also a strong converse (though not as strong as for discrete groups): if $\pi_k E\com G = 0$ for $k = 1, 2, 4$, then $G$ is abelian \cite{AntolinGritschacherVillarreal2023}.

There is a notion of \emph{transitionally commutative principal $G$-bundle} which is classified by $B\com G$ in the same sense $BG$ classifies principal $G$-bundles. In those terms, $E\com$ is the total space of the universal transitionally commutative principal $G$-bundle, which is a mouthful and the reason we typically refer to it using notation.

For a discrete group $G$, there is a useful poset model for $E\com G$ in terms of abelian subgroups, namely consider the following posets of subgroups or cosets thereof, in all cases ordered by inclusion:
\begin{itemize}
    \item $\operatorname{AbCo}(G)$, the poset of all cosets of abelian subgroups of $G$; and
    \item $\operatorname{mAbCo}(G)$, the subposet of $\operatorname{AbCo}(G)$ consisting of cosets $gB$, where $g\in G$ and $B$ is an intersection of maximal abelian subgroups of $G$.
\end{itemize}

There are two homeomorphic ways to turn a poset $P$ into a topological space: one can form the simplicial complex called the \emph{order complex} of $P$, whose vertices are the elements of $P$ and whose simplices are finite non-empty chains in $P$, and then take the geometric realization of that complex; or, one can think $P$ as a category with exactly one morphism $x \to y$ whenever $x \le_P y$, and take the geometric realization of the nerve of said category. For simplicity we shall call the resulting space simply the \emph{geometric realization of $P$}.

\begin{proposition}[{\cite[Proposition 2.4]{AGS}}]\label[proposition]{prop:coset-poset-model}
Given a discrete group $G$, the following CW-complexes are homotopy equivalent:
\begin{enumerate}
    \item $E\com G$;
    \item the geometric realization of $\operatorname{AbCo}(G)$;
    \item the geometric realization of $\operatorname{mAbCo}(G)$.
\end{enumerate}
\end{proposition}

\begin{corollary}[{\cite[Corollary 4.3]{AGS}}]\label[corollary]{cor:countable-wedge-criterion}
Let $G$ be a discrete group, and denote by $Z$ the center of $G$. Assume that for every pair of distinct maximal abelian subgroups $A$ and $B$ of $G$ one has
\[
    A\cap B=Z,
\]
and assume that $Z$ has countably infinite index in $G$. Then
\[
    E\com G\simeq \bigvee_{\mathbb{N}} S^1.
\]
\end{corollary}

\section{$E\com$ and $B\com$ of $\SL_2(R)$ with $R$ an integral domain}\label{sec:ecom-sl2-domain}

Let $R$ be an integral domain and let $K$ be its field of fractions. We will use the following elementary observation to relate abelian subgroups of $\SL_2(R)$ with centralizers.

\begin{lemma}\label[lemma]{lem:matrix-centralizer-polynomial}
Let $M\in M_2(K)$ be a non-scalar matrix. Then the centralizer of $M$ in the ring $M_2(K)$ of $2 \times 2$ matrices with entries in $K$ is given by:
\[
    C_{M_2(K)}(M)=K[M].
\]
In particular, the centralizer of $M$ is a commutative ring.
\end{lemma}

\begin{proof}
We prove the inclusion $C_{M_2(K)}(M)\subseteq K[M]$; the other inclusion is immediate. Let $X\in M_2(K)$ be such that $XM=MX$. Since $M$ is not a scalar matrix, there exists a vector $v\in K^2$ such that $v$ and $Mv$ are linearly independent. Indeed, if $v$ and $Mv$ were dependent for every $v$, then every vector would be an eigenvector for $M$, forcing there to be a unique eigenvalue $\lambda$ and then forcing $M = \lambda I$. Thus $\{v,Mv\}$ is a basis of $K^2$.

Write
\[
    Xv=\alpha v+\beta Mv
\]
for some $\alpha,\beta\in K$. Since $X$ commutes with $M$, we have
\[
    X(Mv)=M(Xv)=M(\alpha v+\beta Mv)=\alpha Mv+\beta M^2v.
\]

Now the operator $\alpha I+\beta M$ has the same value as $X$ on $v$:
\[
    (\alpha I+\beta M)v=\alpha v+\beta Mv=Xv,
\]
and, because it also commutes with $M$, it has the same value as $X$ on $Mv$:
\[
    (\alpha I+\beta M)(Mv)=M((\alpha I+\beta M)v)=M(Xv)=X(Mv).
\]
Therefore $X$ and $\alpha I+\beta M$ agree on the basis $\{v,Mv\}$, and so
\[
    X=\alpha I+\beta M\in K[M].
\]
Thus $C_{M_2(K)}(M)=K[M]$.
\end{proof}

We warn the reader that the previous result is false for larger matrices!

\begin{corollary}\label[corollary]{cor:centralizers-abelian}
If $M\in \SL_2(R)$ is non-central, then $C_{\SL_2(R)}(M)$, the centralizer of $M$ in the group $\SL_2(R)$, is abelian.
\end{corollary}

\begin{proof}
The group $C_{\SL_2(R)}(M)$ is contained in the ring $C_{M_2(K)}(M)=K[M]$. Since $K[M]$ is a commutative ring, the centralizer group in $\SL_2(R)$ is also abelian.
\end{proof}

\begin{corollary}\label[corollary]{cor:maximal-abelian-centralizers}
Every maximal abelian subgroup of $\SL_2(R)$ is the centralizer of any of its non-central elements. More precisely, if $A\leq \SL_2(R)$ is maximal abelian and $M\in A\setminus Z(\SL_2(R))$, then
\[
    A=C_{\SL_2(R)}(M).
\]
Moreover, if $A$ and $B$ are distinct maximal abelian subgroups of $\SL_2(R)$, then
\[
    A\cap B=Z(\SL_2(R)).
\]
\end{corollary}

\begin{proof}
Since $A$ is abelian and contains $M$, we have $A\leq C_{\SL_2(R)}(M)$. By \cref{cor:centralizers-abelian}, $C_{\SL_2(R)}(M)$ is abelian (and is a proper subgroup of $\SL_2(R)$ since $M$ is not central), so the maximality of $A$ gives equality. If $A$ and $B$ are distinct maximal abelian subgroups and $M\in (A\cap B)\setminus Z(\SL_2(R))$, then both $A$ and $B$ are equal to $C_{\SL_2(R)}(M)$, a contradiction. Hence their intersection is contained in the center, and the reverse inclusion is automatic: indeed, every maximal abelian subgroup must contain the center $Z$ because $\langle A \cup Z \rangle$ is also abelian and contains $A$. 
\end{proof}

Let $\calA_{\max}$ denote the collection of maximal abelian subgroups of $\SL_2(R)$, and write
\[
    Z=Z(\SL_2(R)).
\]
\Cref{cor:maximal-abelian-centralizers} shows that distinct members of $\calA_{\max}$ intersect exactly in $Z$. Thus the poset of maximal abelian subgroups and their intersections is $\calA_{\max} \cup \{Z\}$, with minimum element $Z$ contained in all the other elements, which are pairwise incomparable.

We will now use the preceding information to compute the homotopy tyhpes of the spaces $E\com \SL_2(R)$ and $B\com \SL_2(R)$. For these calculations we consider $\SL_2(R)$ as a discrete group.

\begin{theorem}\label[theorem]{prop:ecom-sl2-domain-criterion}
Let $R$ be a countably infinite integral domain. Then
\[
    E\com(\SL_2(R))\simeq \bigvee_{\mathbb{N}} S^1.
\]
\end{theorem}

\begin{proof}
By \cref{cor:maximal-abelian-centralizers}, any two distinct maximal abelian subgroups of $\SL_2(R)$ intersect exactly in the center. The center is $\{\pm I\}$ which is finite, and thus has countably infinite index in $\SL_2(R)$. The claim follows from \cref{cor:countable-wedge-criterion}.
\end{proof}

\begin{theorem}\label[theorem]{thm:bcom-sl2-domain-amalgam}
Let $R$ be a countably infinite integral domain. Then the classifying space $B\com(\SL_2(R))$ is a $K(\Gamma_R,1)$ where the group $\Gamma_R$ is the amalgamated product of all maximal abelian subgroups of $\SL_2(R)$, amalgamated over the center, that is:
\[
    B\com(\SL_2(R))\simeq K(\Gamma_R,1) \quad \text{where} \quad
    \Gamma_R = \underset{A\in\calA_{\max}}{\bigast\nolimits_Z} A.
\]
\end{theorem}

\begin{proof}
In general, for a discrete group one has that \[B\com G \simeq \hocolim_{A \in \operatorname{Ab}(G)} BA,\] where the homotopy colimit is indexed by the poset of abelian subgroups of $G$. We can restrict the homotopy colimit to the subposet of those abelian subgroups which are intersections of maximal abelian subgroups, since that poset is cofinal in the poset of all abelian subgroups. In our case, this smaller poset is a directed tree: it only has the maximal abelian subgroups and the center, which is contained in each of them, and no further inclusions. A homotopy colimit of a diagram of classifying spaces of discrete groups whose shape is the free category on some directed graph and whose morphisms are induced from injective groups homomorphisms, is automatically also a $K(G,1)$ and $G$ is given by the colimit of the corresponding diagram of groups \cite[Theorem 1B.11]{Ha02}.
\end{proof}

 In \cref{sec:prelim-sl2-localized} we specialize these general observations to
\[
    \SL_2(\Z[1/p]),
\]
where the maximal abelian subgroups are precisely the centralizers of non-central elements. Thus the classification of elements into unipotent, split, real irreducible, and imaginary irreducible types leads directly to the list of maximal abelian subgroups. The resulting consequences for $E\com$ and $B\com$ are collected in \cref{sec:consequences-ecom-bcom}.

\section{The geometry of elements of $\SL_2(\Z[1/p])$}\label{sec:prelim-sl2-localized}

The purpose of this section is to explain the geometric information that will be used in the computation of centralizers. There are three ingredients. First, we fix the two ambient actions of $\SL_2(\Z[1/p])$: the Archimedean action on the hyperbolic plane and the non-Archimedean action on the Bruhat--Tits tree $T_p$. Second, we classify elements by the factorization of their characteristic polynomial over $\Q$. Third, we translate that algebraic classification into fixed points, axes, and min-sets in $\mathbb{H}^2\times T_p$. The central point is that an element of the centralizer of $g$ must preserve the geometric objects naturally associated to $g$, and the number of translation directions in those objects will determine the free abelian rank of the centralizer.

Recall that given a group element $g$ and point $x$ in a metric space acted on by the group, we call $d(x, g\cdot x)$ the \emph{displacement of $x$}, and we let the \emph{displacement of $g$} be the infimum of the displacements of all points $x$. The \emph{min-set} of $g$ is the set of all points achieving the infimum, that is, the set of points whose displacement is the displacement of $g$. The min-set of $g$ may very well be empty.

The definitions and basic facts concerning the Bruhat--Tits tree, its boundary, vertex stabilizers, and the elliptic/hyperbolic dichotomy for elements of $\SL_2(\Q_p)$ are standard; we follow the treatments in Serre's book on trees and in Bass's account of tree lattices \cite{SerreTrees}.

Let $p$ be a prime number. We denote by
\[
    \Z[1/p]=\left\{\frac{a}{p^k}\mid a\in \Z,\ k\geq 0\right\}
\]
the localization of $\Z$ obtained by inverting the powers of $p$. The group of interest is
\[
    \SL_2(\Z[1/p])=
    \left\{
        \begin{pmatrix}
            a & b \\
            c & d
        \end{pmatrix}
        \in M_2(\Z[1/p])\mid ad-bc=1
    \right\}.
\]
This group contains $\SL_2(\Z)$ as a subgroup and can be viewed simultaneously as a subgroup of $\SL_2(\R)$, via the inclusion $\Z[1/p]\subset \R$, and as a subgroup of $\SL_2(\Q_p)$, via the inclusion $\Z[1/p]\subset \Q_p$.

The inclusion $\SL_2(\Z[1/p])\leq \SL_2(\R)$ induces an action by isometries on the hyperbolic plane
\[
    \mathbb{H}^2=\{z\in \mathbb{C}\mid \operatorname{Im}(z)>0\}.
\]
Explicitly, a matrix $\begin{pmatrix}a&b\\ c&d\end{pmatrix}$ acts by M\"obius transformations,
\[
    z\longmapsto \frac{az+b}{cz+d}.
\]
Since the determinant is one, this action preserves orientation and is realized by hyperbolic isometries.

On the other hand, the inclusion $\SL_2(\Z[1/p])\leq \SL_2(\Q_p)$ allows us to consider its action on the Bruhat--Tits tree associated to $\SL_2(\Q_p)$, which we denote by $T_p$. We recall the construction of this tree in \cref{subsec:bruhat-tits-eigendirections}; for now, we only use that the natural action of $\SL_2(\Q_p)$ on $\Q_p^2$ induces a simplicial action on $T_p$, and hence an action of $\SL_2(\Z[1/p])$ by restriction.

Consequently, $\SL_2(\Z[1/p])$ acts diagonally on the product $\mathbb{H}^2\times T_p$. This action combines the Archimedean geometry of the hyperbolic plane with the non-Archimedean geometry of the Bruhat--Tits tree, and it will be the starting point for studying the structure of centralizers of elements in $\SL_2(\Z[1/p])$.

\subsection{Classification of elements of $\SL_2(\Z[1/p])$}\label{subsec:classification-elements}

Let $g\in \SL_2(\Z[1/p])$. Its characteristic polynomial is
\[
    \chi_g(x)=x^2-\tr(g)x+1,
\]
and its discriminant is
\[
    \Delta(g)=\tr(g)^2-4.
\]
The following classification is determined by the factorization of $\chi_g(x)$ over $\Q$ and, in the irreducible case, by the sign of $\Delta(g)$.

\begin{definition}\label[definition]{def:unipotent-element}
An element $g\in \SL_2(\Z[1/p])$ is called \emph{unipotent} if its characteristic polynomial has a repeated root equal to $1$ or $-1$. Equivalently,
\[
    \tr(g)=2 \quad \text{or} \quad \tr(g)=-2.
\]
Thus $g$ has a unique eigenvalue $\epsilon\in\{\pm 1\}$, counted with multiplicity. For example,
\[
    \begin{pmatrix}
        1 & 1 \\
        0 & 1
    \end{pmatrix}
    \in \SL_2(\Z[1/p])
\]
is unipotent, since its characteristic polynomial is $(x-1)^2$.
\end{definition}

\begin{definition}\label[definition]{def:split-element}
An element $g\in \SL_2(\Z[1/p])$ is called \emph{split} if its characteristic polynomial has two distinct roots in $\Q$. Equivalently, $\Delta(g)$ is a nonzero square in $\Q$. In this case $g$ is diagonalizable over $\Q$, and its two eigenvalues are rational and inverse to each other. For example,
\[
    \begin{pmatrix}
        p & 0 \\
        0 & p^{-1}
    \end{pmatrix}
    \in \SL_2(\Z[1/p])
\]
is split, since its eigenvalues are $p$ and $p^{-1}$.
\end{definition}

\begin{definition}\label[definition]{def:real-irreducible-element}
An element $g\in \SL_2(\Z[1/p])$ is called \emph{real irreducible} if its characteristic polynomial is irreducible over $\Q$ and has two real roots. Equivalently, $\Delta(g)>0$ and $\Delta(g)$ is not a square in $\Q$. For example,
\[
    \begin{pmatrix}
        0 & -1 \\
        1 & 3
    \end{pmatrix}
    \in \SL_2(\Z[1/p])
\]
is real irreducible, since its characteristic polynomial is $x^2-3x+1$, whose discriminant is $5$.
\end{definition}

\begin{definition}\label[definition]{def:imaginary-irreducible-element}
An element $g\in \SL_2(\Z[1/p])$ is called \emph{imaginary irreducible} if its characteristic polynomial is irreducible over $\Q$ and has no real roots. Equivalently, $\Delta(g)<0$. For example,
\[
    \begin{pmatrix}
        0 & -1 \\
        1 & 0
    \end{pmatrix}
    \in \SL_2(\Z[1/p])
\]
is imaginary irreducible, since its characteristic polynomial is $x^2+1$, whose discriminant is $-4$.
\end{definition}

Note that Definitions~\ref{def:unipotent-element}--\ref{def:imaginary-irreducible-element} exhaust all elements of $\SL_2(\Z[1/p])$. Indeed, $\chi_g(x)$ either has a repeated root, splits over $\Q$ with distinct roots, or is irreducible over $\Q$; in the irreducible case the sign of $\Delta(g)$ determines whether the roots are real or non-real complex conjugates.\\

When $g$ is irreducible over $\Q$ and $\lambda$ is an eigenvalue, we write $K=\Q(\lambda)$. Note that this field automatically contains the other eigenvalue as well, and is the splitting field of the characteristic polynomial of $g$.

The next issue is local. The behavior of $g$ on the Bruhat--Tits tree is governed by whether the prime $p$ splits in the quadratic field $K$. This local distinction is independent of the distinction between the real irreducible and imaginary irreducible cases, but it is exactly what decides whether the $p$-adic factor contributes an additional translation direction to the centralizer.

\subsection{The Bruhat--Tits tree and eigendirections}\label{subsec:bruhat-tits-eigendirections}

We recall in more detail the features of the action on the Bruhat--Tits tree $T_p$ that will be used in the computation of centralizers. A \emph{lattice} in $\Q_p^2$ is a free $\Z_p$-submodule of rank two which spans $\Q_p^2$ over $\Q_p$. Two lattices $L$ and $L'$ are homothetic if $L'=aL$ for some $a\in \Q_p^*$. The vertices of $T_p$ are homothety classes $[L]$ of lattices. Two vertices $[L]$ and $[L']$ are joined by an edge if one can choose representatives so that
\[
    pL\subsetneq L'\subsetneq L.
\]
Equivalently, $L/L'$ is a one-dimensional vector space over $\mathbb{F}_p$. With this convention every vertex has valence $p+1$, and $T_p$ is a tree.

The group $\GL_2(\Q_p)$ acts on lattices by its linear action on $\Q_p^2$, and scalar matrices act trivially on homothety classes. Hence the action factors through $\PGL_2(\Q_p)$; in particular $\SL_2(\Q_p)$, and therefore $\SL_2(\Z[1/p])$, acts simplicially on $T_p$. The stabilizer in $\SL_2(\Q_p)$ of the standard vertex $[\Z_p^2]$ is $\SL_2(\Z_p)$. More generally, stabilizers of vertices are compact open subgroups of $\SL_2(\Q_p)$.

The boundary $\partial T_p$ can be identified with the projective line $\mathbb{P}^1(\Q_p)$. To make this explicit, an end of the tree may be represented by an infinite geodesic ray
\[
    [L_0],[L_1],[L_2],\ldots
\]
starting at some vertex. After choosing representatives, one may arrange that the lattices are nested,
\[
    L_0\supsetneq L_1\supsetneq L_2\supsetneq \cdots,
\]
with each inclusion corresponding to one edge. This is what we mean by a decreasing ray of lattices. Such a ray progressively singles out one direction in $\Q_p^2$: the intersection
\[
    \bigcap_{n\geq 0} L_n
\]
is a rank-one $\Z_p$-module in the direction followed by the ray, and its $\Q_p$-span is a unique one-dimensional $\Q_p$-subspace. Equivalently, a ray converging to a boundary point determines a line $\ell\leq \Q_p^2$. It can be seen that this correspondence gives a bijection $\partial T_p \to \mathbb{P}^1(\Q_p)$. Thus the corresponding point of $\mathbb{P}^1(\Q_p)$ is the unique $\Q_p$-line selected by that ray. In this way, fixed points of an element of $\SL_2(\Q_p)$ on $\partial T_p$ are the same as invariant lines in $\Q_p^2$.

We will refer to these invariant lines as eigendirections. More precisely, if $g\in \SL_2(\Z[1/p])$ and $F$ is a field containing $\Q$, an \emph{$F$-eigendirection} of $g$ is a point $\ell\in \mathbb{P}^1(F)$ such that $g\ell=\ell$. Equivalently, $\ell$ is the line spanned by an eigenvector of $g$ in $F^2$. If $\lambda$ is the corresponding eigenvalue, then the eigendirection is the kernel of $g-\lambda I$ in $F^2$. Thus $g$ has two distinct $\Q_p$-eigendirections exactly when its characteristic polynomial splits over $\Q_p$ with two distinct roots.

The dynamical type of an element of $\SL_2(\Q_p)$ on $T_p$ is governed by these eigendirections. For the elements of $\SL_2(\Z[1/p])$ that appear in the classification below, this $p$-adic dynamics should be read together with the factorization of the characteristic polynomial over $\Q$.

First, a non-central unipotent element has a single eigendirection, defined already over $\Q$ and hence over $\Q_p$. Its unique eigenvalue is $\pm 1$, so its $p$-adic translation length is zero. Thus it fixes a nonempty subtree of $T_p$; geometrically, the unipotent direction that appears in its centralizer comes from the parabolic action on $\mathbb{H}^2$, not from a translation axis in $T_p$.

Second, if $g$ is split over $\Q$ with two distinct rational eigendirections, then these two eigendirections are also defined over $\Q_p$. Hence $g$ preserves the geodesic in $T_p$ with these two endpoints. After conjugating in $\GL_2(\Q_p)$, such an element has the form
\[
    \begin{pmatrix}
        \lambda & 0 \\
        0 & \lambda^{-1}
    \end{pmatrix},
\]
and its translation length on $T_p$ is
\[
    \ell_{T_p}(g)=2\left|v_p(\lambda)\right|.
\]
It is hyperbolic on $T_p$ exactly when $v_p(\lambda)\neq 0$; if $v_p(\lambda)=0$, it fixes a vertex but still preserves the geodesic determined by its two boundary fixed points. For elements in $\SL_2(\Z[1/p])$, this is the $p$-adic reflection of the localization direction.

Third, suppose that $g$ is irreducible over $\Q$ and let $K=\Q(\lambda)$ be the quadratic field generated by an eigenvalue. The distinction relevant to $T_p$ is whether $p$ splits in $K$. If $p$ splits in $K$, then the two embeddings $K\hookrightarrow \Q_p$ give two distinct $\Q_p$-eigendirections, so $g$ preserves the corresponding geodesic in $T_p$ and may contribute a $p$-adic translation direction. If $p$ does not split in $K$, then the characteristic polynomial remains irreducible over $\Q_p$; there are no $\Q_p$-eigendirections, no fixed point on $\partial T_p$, and $g$ is elliptic on $T_p$, fixing a nonempty subtree.

This description is compatible with the four algebraic types above. In the unipotent case, the $p$-adic action is elliptic and the infinite part of the centralizer comes from the unipotent direction. In the split-over-$\Q$ case, the two rational eigendirections give a preserved geodesic in both the real and $p$-adic settings, but the allowed translations are controlled by the single unit group $\Z[1/p]^\times$. In the real irreducible case, the real factor always supplies a hyperbolic axis, while the $p$-adic factor supplies an additional independent axis exactly when $p$ splits in $K$. In the imaginary irreducible case, the real factor is elliptic, so an infinite cyclic direction appears only in the split-at-$p$ case.

The following summary records the geometric object that the centralizer must preserve in each case. Here $A_\infty$ denotes the real hyperbolic axis when it exists, $A_p$ denotes the axis in $T_p$ when it exists, and $x_g\in\mathbb{H}^2$ denotes the fixed point of an elliptic element.
\[
\resizebox{\textwidth}{!}{$
\begin{array}{c|c|c|c}
\text{Type of }g & \mathbb{H}^2 & T_p & \text{Object preserved by }C_{\SL_2(\Z[1/p])}(g) \\
\hline
\text{unipotent} & \text{parabolic fixed point} & \text{elliptic} & \text{boundary point and horocycles} \\
\text{split over }\Q & A_\infty & \text{geodesic from eigendirections} & \text{real and }p\text{-adic eigendirections} \\
\text{real irreducible, }p\text{ does not split in }K & A_\infty & \operatorname{Fix}_{T_p}(g) & A_\infty\times \operatorname{Fix}_{T_p}(g) \\
\text{real irreducible, }p\text{ split in }K & A_\infty & A_p & A_\infty\times A_p \\
\text{imaginary irreducible, }p\text{ does not split in }K & x_g & \operatorname{Fix}_{T_p}(g) & \{x_g\}\times \operatorname{Fix}_{T_p}(g) \\
\text{imaginary irreducible, }p\text{ split in }K & x_g & A_p & \{x_g\}\times A_p
\end{array}
$}
\]

Thus the arithmetic condition ``$p$ splits in $K$'' has a precise geometric meaning: it is equivalent to the existence of two $\Q_p$-eigendirections, hence to the existence of a $p$-adic axis that may contribute a translation direction. When there are no $\Q_p$-eigendirections, the $p$-adic action is elliptic and contributes no such direction. In the product $\mathbb{H}^2\times T_p$, the centralizer preserves the product of the min-sets in the two factors; consequently the number of independent translation directions visible in these min-sets controls the free abelian rank appearing in the centralizer. This is the organizing principle for the computations in \cref{sec:centralizers}.

\section{Centralizers of elements in $\SL_2(\Z[1/p])$ }\label{sec:centralizers}

We now organize the centralizers of non-central elements of $\SL_2(\Z[1/p])$ according to the classification in \cref{subsec:classification-elements}. The central elements are $\pm I$, and their centralizer is the whole group $\SL_2(\Z[1/p])$. Thus, unless otherwise stated, we assume throughout this section that $g\neq \pm I$.

The centralizer of $g$ will be denoted by
\[
    C_{\SL_2(\Z[1/p])}(g)=\{h\in \SL_2(\Z[1/p])\mid hg=gh\}.
\]

The following table gives the expected algebraic type of the centralizer in each case. Here $K=\Q(\lambda)$ denotes the quadratic algebra generated by an eigenvalue of $g$, and saying that $p$ \emph{splits in $K$} means that the prime $p$ splits in the quadratic field $K$.

\[
\resizebox{\textwidth}{!}{$
\begin{array}{c|c|c|c|c|c}
\text{Type of }g & \text{In }\mathbb{H}^2 & \text{In }T_p & \text{Arithmetic condition} & \text{Example} & C_{\SL_2(\Z[1/p])}(g) \\
\hline
\text{Unipotent} & \text{parabolic} & \text{elliptic} & \text{always} &
\begin{pmatrix}1&1\\0&1\end{pmatrix} & C_2\times \Z[1/p] \\
\hline
\text{Split over } \mathbb Q & \text{hyperbolic} & \text{hyperbolic} & \text{always} &
\begin{pmatrix}p&0\\0&p^{-1}\end{pmatrix} & C_2\times \Z \\
\hline
\text{Real irreducible} & \text{hyperbolic} & \text{elliptic} & p\text{ does not split in }K &
\begin{pmatrix}0&-1\\1&3\end{pmatrix} & C_2\times \Z \\
\hline
\text{Real irreducible} & \text{hyperbolic} & \text{hyperbolic} & p\text{ splits in }K &
\begin{pmatrix}0&-1\\1&3\end{pmatrix},\ p=7 & C_2\times \Z^2 \\
\hline
\text{Imaginary irreducible} & \text{elliptic} & \text{elliptic} & p\not\equiv 1\pmod 4,
K=\Q(i) &
\begin{pmatrix}0&-1\\1&0\end{pmatrix} & C_4 \\
\hline
\text{Imaginary irreducible} & \text{elliptic} & \text{non-compact fixed direction} & p\equiv 1\pmod 4,
K=\Q(i) &
\begin{pmatrix}0&-1\\1&0\end{pmatrix} & C_4\times \Z \\
\hline
\text{Imaginary irreducible} & \text{elliptic} & \text{elliptic} & p=3\text{ or }p\equiv 2\pmod 3,
K=\Q(\sqrt{-3}) &
\begin{pmatrix}0&-1\\1&1\end{pmatrix} & C_6 \\
\hline
\text{Imaginary irreducible} & \text{elliptic} & \text{non-compact fixed direction} & p\equiv 1\pmod 3,
K=\Q(\sqrt{-3}) &
\begin{pmatrix}0&-1\\1&1\end{pmatrix} & C_6\times \Z
\end{array}
$}
\]

\subsection{The unipotent case}\label{subsec:unipotent-centralizers}

Let $g\in \SL_2(\Z[1/p])$ be non-central unipotent. Then
\[
    g=\epsilon I+N,
    \qquad \epsilon\in\{\pm 1\},
\]
where $N\neq 0$ and $N^2=0$. Since the scalar part commutes with every matrix, the centralizer of $g$ is the same as the centralizer of $N$.

We first describe the geometry. The nilpotent matrix $N$ has a one-dimensional kernel, and this line is the unique eigendirection of $g$. Over $\R$ this eigendirection determines the unique fixed point of $g$ on $\partial\mathbb{H}^2$. Thus $g$ acts parabolically on $\mathbb{H}^2$. Any element commuting with $g$ preserves this unique boundary point, and hence preserves the corresponding family of horocycles. This is the geometric source of the unipotent direction in the centralizer.

The same eigendirection is defined over $\Q_p$, but the repeated eigenvalue is $\epsilon\in\{\pm 1\}$ and has $p$-adic valuation zero. Consequently $g$ has zero translation length on $T_p$. Equivalently, after choosing a basis adapted to the eigendirection, $g$ lies in a conjugate of $\SL_2(\Z_p)$, and therefore fixes a vertex of $T_p$. Thus the $p$-adic factor contributes a fixed subtree rather than a translation axis. The min-set in the product has the form of a parabolic horocyclic direction in $\mathbb{H}^2$ together with a fixed subtree in $T_p$.

We now compute the centralizer. Since $N$ is a nonzero nilpotent $2\times 2$ matrix, it is conjugate over $\Q$ to a matrix of the form
\[
    E=\begin{pmatrix}0&u\\0&0\end{pmatrix},
\]
for some $u \in \Q^\times$.
This is only a rational normal form; the centralizer in $\SL_2(\Z[1/p])$ is obtained by intersecting the rational centralizer with the integral group. If
\[
    A=\begin{pmatrix}a&b\\c&d\end{pmatrix},
\]
then the equation $AE=EA$ gives $c=0$ and $a=d$. Hence
\[
    C_{\SL_2(\Q)}(E)=
    \left\{\pm\begin{pmatrix}1&t\\0&1\end{pmatrix}\mid t\in \Q\right\},
\]
and, equivalently,
\[
    C_{\SL_2(\Q)}(N)=\{\pm(I+tN)\mid t\in \Q\}.
\]
Intersecting with $\SL_2(\Z[1/p])$ gives
\[
    C_{\SL_2(\Z[1/p])}(g)=\{\pm(I+tN)\mid tN\in M_2(\Z[1/p])\}.
\]
Let
\[
    L_N=\{t\in\Q\mid tN\in M_2(\Z[1/p])\}.
\]
This is a nonzero fractional ideal of the principal ideal domain $\Z[1/p]$: explicitly it is the finite intersection of the modules $n_{ij}^{-1}\Z[1/p]$, where $n_{ij}$ runs over the nonzero entries of $N$. Therefore $L_N=r\Z[1/p]$ for some $r\in\Q^*$, and as an additive group it is isomorphic to $\Z[1/p]$. Hence we have a group isomorphism:
\[
    C_{\SL_2(\Z[1/p])}(g)\cong C_2\times \Z[1/p].
\]

\subsection{The split case}\label{subsec:split-centralizers}

Let $g\in \SL_2(\Z[1/p])$ be split. Then the characteristic polynomial of $g$ has two distinct roots in $\Q$, say $\lambda$ and $\lambda^{-1}$, with $\lambda\in\Q^*$ and $\lambda\neq \pm1$. Thus $g$ has two rational eigendirections.

In the hyperbolic plane these two eigendirections are two points of $\partial\mathbb{H}^2$. Hence $g$ is hyperbolic on $\mathbb{H}^2$ and acts by translation along the geodesic $A_\infty$ joining them. Every element of the centralizer preserves the two endpoints of $A_\infty$; since it commutes with $g$, it preserves each endpoint individually and acts on $A_\infty$ by translations.

The same two eigendirections are defined over $\Q_p$. Therefore $g$ preserves the corresponding geodesic $A_p\subset T_p$. In the diagonal model
\[
    \begin{pmatrix}
        \lambda & 0 \\
        0 & \lambda^{-1}
    \end{pmatrix},
\]
the translation length on $T_p$ is $2|v_p(\lambda)|$. For elements of $\SL_2(\Z[1/p])$ which are split and non-central, the possible denominators are powers of $p$, so the $p$-adic component detects precisely the localization direction. In the prototypical case $\diag(p,p^{-1})$, the axis $A_p$ is translated nontrivially.

Thus the centralizer preserves the product of the two axes
\[
    A_\infty\times A_p\subset \mathbb{H}^2\times T_p.
\]
However, in the rationally split case these two translation directions are controlled by the same unit group of $\Z[1/p]$, rather than by two independent unit groups. This is why the centralizer has one infinite cyclic factor, not two.

We now make this algebraic. Choose $P\in \GL_2(\Q)$ such that
\[
    PgP^{-1}=\begin{pmatrix}\lambda&0\\0&\lambda^{-1}\end{pmatrix}.
\]
A matrix commuting with this diagonal matrix must preserve the two eigenspaces. Hence
\[
    C_{\SL_2(\Q)}(g)=
    P^{-1}\left\{\begin{pmatrix}a&0\\0&a^{-1}\end{pmatrix}\mid a\in \Q^*\right\}P.
\]
Thus the rational centralizer is the split norm-one torus.

Let $R=\Z[1/p]$. The centralizer inside $\SL_2(R)$ is the intersection of this rational torus with $\SL_2(R)$. Since
\[
    R[g]\subseteq \Q[g]\cong \Q\times\Q
\]
is an $R$-order in the split algebra, we have
\[
    C_{\SL_2(R)}(g)=\{u\in \Q[g]\cap M_2(R)\mid N(u)=1\}.
\]
This is the norm-one unit group of a split $R$-order, and it is commensurable with
\[
    \{(a,a^{-1})\mid a\in R^*\}.
\]
Since
\[
    R^*=\Z[1/p]^*=\{\pm p^n\mid n\in\Z\}\cong C_2\times \Z,
\]
and the central element $-I$ belongs to the centralizer, we obtain
\[
    C_{\SL_2(\Z[1/p])}(g)\cong C_2\times \Z.
\]
For example, if
\[
    s=\begin{pmatrix}p&0\\0&p^{-1}\end{pmatrix},
\]
then
\[
    C_{\SL_2(\Z[1/p])}(s)=
    \left\{\begin{pmatrix}\pm p^n&0\\0&\pm p^{-n}\end{pmatrix}\mid n\in\Z\right\}
    \cong C_2\times \Z.
\]

\subsection{The real irreducible case}\label{subsec:real-irreducible-centralizers}

Let $g\in \SL_2(\Z[1/p])$ be real irreducible. Then $\chi_g(x)$ is irreducible over $\Q$ and has two real roots. If $\lambda$ is one of these roots, then
\[
    K=\Q(\lambda)
\]
is a real quadratic field. Under the standard identification
\[
    \partial\mathbb{H}^2\cong \mathbb{P}^1(\R),
\]
the fixed points of the M\"obius action of $g$ on $\partial\mathbb{H}^2$ are precisely the real eigendirections of $g$. Indeed, a line $[v]\in \mathbb{P}^1(\R)$ is fixed by the projective action of $g$ if and only if $gv$ is a scalar multiple of $v$, that is, if and only if $[v]$ is an eigendirection. Since $g$ has two distinct real eigenvalues, it has two fixed points on $\partial\mathbb{H}^2$ and acts hyperbolically on $\mathbb{H}^2$. We denote by $A_\infty$ the geodesic joining these two fixed points.

Every element of $C_{\SL_2(\Z[1/p])}(g)$ preserves the set of fixed points of $g$ on $\partial\mathbb{H}^2$: if $h$ commutes with $g$ and $x$ is fixed by $g$, then
\[
    g(hx)=h(gx)=hx.
\]
Moreover, the two endpoints of $A_\infty$ cannot be interchanged by an element commuting with $g$. They are the attracting and repelling fixed points of the hyperbolic isometry $g$; interchanging them would conjugate $g$ to $g^{-1}$, whereas an element of the centralizer conjugates $g$ to itself. Hence every element of the centralizer preserves each endpoint of $A_\infty$ individually and acts on $A_\infty$ by translations.

We now turn to the $p$-adic factor. The relevant distinction is whether $p$ splits in the real quadratic field $K$. If $p$ does not split in $K$, then
\[
    K\otimes_\Q \Q_p
\]
is a field. Equivalently, the characteristic polynomial of $g$ remains irreducible over $\Q_p$, so the two eigendirections of $g$ are not points of $\mathbb{P}^1(\Q_p)=\partial T_p$. Thus $g$ has no fixed point on $\partial T_p$. Since a hyperbolic automorphism of a tree has exactly two fixed points on the boundary, $g$ cannot be hyperbolic on $T_p$; therefore it is elliptic and fixes a nonempty subtree $\operatorname{Fix}_{T_p}(g)$. This fixed point set is a subtree: if two vertices are fixed, then the unique geodesic segment joining them is fixed pointwise.

In this nonsplit case the min-set in the product has the form
\[
    \Min(g)=A_\infty\times \operatorname{Fix}_{T_p}(g).
\]
The centralizer preserves this min-set. Its action on the first factor gives a homomorphism
\[
    \tau\colon C_{\SL_2(\Z[1/p])}(g)\longrightarrow \operatorname{Trans}(A_\infty)\cong \R.
\]
The kernel fixes $A_\infty$ pointwise and also fixes a point of $\operatorname{Fix}_{T_p}(g)$, hence fixes a point of $\mathbb{H}^2\times T_p$; by properness of the action, this kernel is finite. Properness also implies that the image of $\tau$ is discrete. Since $g$ itself translates nontrivially along $A_\infty$, the image is infinite cyclic. Thus the centralizer is virtually cyclic; because the centralizer is abelian and contains the central element $-I$, its type is
\[
    C_{\SL_2(\Z[1/p])}(g)\cong C_2\times \Z.
\]

If $p$ splits in $K$, then
\[
    K\otimes_\Q \Q_p\cong \Q_p\times \Q_p,
\]
and the two embeddings $K\hookrightarrow \Q_p$ determine two distinct $\Q_p$-eigendirections. Hence $g$ preserves the geodesic $A_p\subset T_p$ with these two endpoints. The element $g$ acts hyperbolically on $T_p$ exactly when the corresponding eigenvalue has nonzero $p$-adic valuation; if that valuation is zero, $g$ fixes vertices of $T_p$ while still preserving the geodesic $A_p$. Thus the correct conclusion from the splitting of $p$ is the existence of the invariant $p$-adic geodesic, not automatically positive translation length for the particular element $g$.

The algebraic centralizer makes the rank clear. Since $g$ is non-scalar and irreducible over $\Q$, we have
\[
    C_{\SL_2(\Q)}(g)=\{u\in K^*\mid N_{K/\Q}(u)=1\},
\]
under the embedding $K=\Q[g]\subset M_2(\Q)$. Intersecting with $\SL_2(\Z[1/p])$ gives the norm-one group of an order in the localized ring $\calO_K[1/p]$. By the $S$-unit theorem \cite[Chapter I, Section 7]{NeukirchANT}, this norm-one group has rank one when $p$ does not split in $K$ and rank two when $p$ splits in $K$. In the split case the second generator may be understood geometrically as a translation in the $p$-adic direction associated to the two primes of $K$ above $p$. Therefore
\[
    C_{\SL_2(\Z[1/p])}(g)\cong C_2\times \Z^2
\]
when $p$ splits in $K$.

As an example, consider
\[
    r=\begin{pmatrix}0&-1\\1&3\end{pmatrix}.
\]
Its characteristic polynomial is $x^2-3x+1$, so $K=\Q(\sqrt{5})$. If $p$ does not split in $\Q(\sqrt{5})$, then $C_{\SL_2(\Z[1/p])}(r)$ has type $C_2\times \Z$. If $p$ splits in $\Q(\sqrt{5})$, for instance $p=7$, then the centralizer has type $C_2\times \Z^2$.

\subsection{The imaginary irreducible case}\label{subsec:imaginary-irreducible-centralizers}

Let $g\in \SL_2(\Z[1/p])$ be imaginary irreducible. Then $\chi_g(x)$ is irreducible over $\Q$ and has no real roots. If $\lambda$ is an eigenvalue, then
\[
    K=\Q(\lambda)
\]
is an imaginary quadratic field. In the hyperbolic plane the element $g$ is elliptic. Its min-set in the real factor is the single fixed point $x_g\in \mathbb{H}^2$, rather than a geodesic line. Consequently the real factor contributes no translation direction to the centralizer. This is the main difference from the real irreducible case: any infinite-order translation direction must come from the action on the Bruhat--Tits tree.

Suppose first that $p$ does not split in $K$. Then $K\otimes_\Q\Q_p$ is a field, so the eigenlines of $g$ are not defined over $\Q_p$. Hence $g$ has no fixed point on
\[
    \partial T_p=\mathbb{P}^1(\Q_p).
\]
A hyperbolic automorphism of a tree always has two fixed points on the boundary, so $g$ cannot be hyperbolic on $T_p$. Therefore $g$ is elliptic and fixes a nonempty subtree of $T_p$. The min-set of $g$ in $\mathbb{H}^2\times T_p$ has no line factor: it is contained in
\[
    \{x_g\}\times \operatorname{Fix}_{T_p}(g).
\]
Thus there is no source for a free abelian translation direction. Algebraically, the centralizer is the group of norm-one units in an order of the imaginary quadratic field $K$ localized at a nonsplit prime; this has rank zero. Hence the centralizer is finite.

If $p$ splits in $K$, then the two eigenlines of $g$ are defined over $\Q_p$. Thus $g$ preserves the geodesic $A_p\subset T_p$ with these two endpoints. As in the real irreducible case, splitting gives the invariant $p$-adic geodesic; the actual translation length of the particular element $g$ is governed by the $p$-adic valuation of an eigenvalue. The localized norm-one unit group has rank one by the $S$-unit theorem \cite[Chapter I, Section 7]{NeukirchANT}, because the two primes of $K$ above $p$ provide one independent $S$-unit direction. Geometrically, this is the translation direction along $A_p$. Therefore the centralizer is its finite elliptic part times one copy of $\Z$.

The finite elliptic part is usually $\{\pm I\}\cong C_2$. The two exceptional imaginary quadratic fields are
\[
    \Q(i), \qquad \Q(\sqrt{-3}),
\]
whose unit groups contain roots of unity of orders $4$ and $6$, respectively. These are the exceptional cases displayed in the table.

For $K=\Q(i)$, take
\[
    a=\begin{pmatrix}0&-1\\1&0\end{pmatrix}.
\]
Then $a^2=-I$ and $a$ has order four. If $p\not\equiv 1\pmod 4$, then $p$ does not split in $\Q(i)$, and
\[
    C_{\SL_2(\Z[1/p])}(a)\cong C_4.
\]
If $p\equiv 1\pmod 4$, then $p$ splits in $\Q(i)$, and the localization contributes one infinite cyclic direction:
\[
    C_{\SL_2(\Z[1/p])}(a)\cong C_4\times \Z.
\]

For $K=\Q(\sqrt{-3})$, take
\[
    b=\begin{pmatrix}0&-1\\1&1\end{pmatrix}.
\]
The characteristic polynomial of $b$ is $x^2-x+1$, and $b$ has order six. If $p=3$ or $p\equiv 2\pmod 3$, then $p$ does not split in $\Q(\sqrt{-3})$, and
\[
    C_{\SL_2(\Z[1/p])}(b)\cong C_6.
\]
If $p\equiv 1\pmod 3$, then $p$ splits in $\Q(\sqrt{-3})$, and
\[
    C_{\SL_2(\Z[1/p])}(b)\cong C_6\times \Z.
\]

More generally, for an imaginary quadratic field $K$ different from $\Q(i)$ and $\Q(\sqrt{-3})$, the same argument gives a finite part $C_2$ instead of $C_4$ or $C_6$, with an additional $\Z$ factor exactly when $p$ splits in $K$.

\section{The spaces $E\com$ and $B\com$ for the discrete group $\SL_2(\Z[1/p])^\delta$}\label{sec:consequences-ecom-bcom}

In this section we apply the preceding algebraic computations to the discrete group
\[
    G=\SL_2(\Z[1/p])^{\delta}.
\]
From now on, $\delta$ will mean that the corresponding spaces is endowed with the discrete topology. Since $\Z[1/p]$ is a countably infinite integral domain, theorems \ref{prop:ecom-sl2-domain-criterion} and \ref{thm:bcom-sl2-domain-amalgam} directly give us a description of $E\com \SL_2(\Z[1/p])^\delta$ and $B\com \SL_2(\Z[1/p])^\delta$. The description of $E\com \SL_2(\Z[1/p])^\delta$ is as explicit as one could want, so we simply restate it here. However, the general description of $E\com \SL_2(\Z[1/p])^\delta$ obtained in Section \ref{sec:ecom-sl2-domain} gives it in terms of the centralizers of non-central elements of $\SL_2(\Z[1/p])$, which we have computed in Section \ref{sec:centralizers}; so out statement here can be more explicit.

\begin{theorem}\label[theorem]{thm:ecom-sl2-localized-wedge}
For every prime $p$,
\[
    E\com\bigl(\SL_2(\Z[1/p])^{\delta}\bigr)
    \simeq
    \bigvee_{\mathbb{N}} S^1.
\]
\end{theorem}

\begin{proof}
This is simply Theorem \ref{prop:ecom-sl2-domain-criterion} in the special case of $R = \SL_2(\Z[1/p])$. (In Section \ref{sec:ecom-sl2-domain} we gave $\SL_2(R)$ the discrete topology.)
\end{proof}

\begin{theorem}\label[theorem]{thm:bcom-sl2-localized-amalgam}
For every prime $p$,
\[
    B\com SL_2(\Z[1/p])^{\delta}\simeq K(\Gamma_p,1),
\]
where $\Gamma_p$ is the amalgamated product of countably many copies of each of the following groups, amalgated over the obvious $C_2$ subgroup coming form the first factors: $C_2\times \Z[1/p]$, $C_2\times \Z$, $C_2\times \Z^2$, $C_2$, $C_4$, $C_6$, $C_4 \times \Z$, $C_6 \times \Z$. 
\end{theorem}

\begin{proof}
This is the specialization of \cref{thm:bcom-sl2-domain-amalgam} to $R=\Z[1/p]$, combined with the description of the centralizers of non-central elements of $\SL_2(\Z[1/p])$ from Section \ref{sec:centralizers}; namely, the centralizers are the unipotent centralizers $C_2\times \Z[1/p]$, the split centralizers $C_2\times \Z$, the real irreducible centralizers $C_2\times \Z$ or $C_2\times \Z^2$ according to the splitting of $p$ in the associated real quadratic field, and the imaginary irreducible centralizers with finite part $C_2$, $C_4$, or $C_6$, possibly multiplied by one copy of $\Z$ when $p$ splits in the associated imaginary quadratic field. Each of these isomorphism types appears countably many times in the amalgamated product $\Gamma_p$.
\end{proof}

Observe that the homotopy type of $E\com \SL_2(\Z[1/p])^\delta$ does not depend on $\Z[1/p]$ beyond it being a countably infinite integral domain, while the description of $B\com \SL_2(Z[1/p])^\delta$ records the actual isomorphism types of the centralizers of non-central elements through the above amalgam.

\section{The topological spaces $E\com$ and $B\com$ for $\SL_2(\Z[1/p])$}\label{sec:topologies}

In this section we record the comparison suggested by the preceding discrete computations and by the topology induced from the inclusion $\Z[1/p]\subset \R$, where $\R$ has its usual topology. We want to compare
$B\com\SL_2(\Z[1/p])$ with $B\com\SL_2(\Z[1/p])^\delta$, and similarly
$E\com\SL_2(\Z[1/p])$ with $E\com\SL_2(\Z[1/p])^\delta$.

It is probably fair to say that when the initial definitions of $B\com G$ and $E\com G$ were given in \cite{AdemCohenTorresGiese}, the authors really had in mind topological groups with well-behaved topologies, mainly Lie groups and discrete groups (which, when countable, are also Lie groups). The results in that paper and the subsequent research on these spaces have certainly focused on those cases. For groups with more complicated topologies, such as $\SL_2(\Z[1/p])$, one might want to revisit a small detail of the definition: the choice of the geometric realization. Recall that both spaces are defined as geometric realizations of simplicial spaces. There is an alternative way to construct a topological space from a simplicial space, called the \emph{fat realization}, which quotients out only by the action of the face maps (instead of using both face and degeneracy maps). Generally speaking, the fat realization is better behaved from the homotopical point of view, but at the same time is not too different from the geometric realization, since for nice enough simplicial spaces it is homotopically equivalent to it. Next, we shall recall two standard facts that make this precise.

First, if $f\colon X_\bullet\to Y_\bullet$ is a levelwise weak equivalence of semisimplicial spaces, then the induced map on fat realizations
\[
    \|X_\bullet\|\longrightarrow \|Y_\bullet\|
\]
is a weak equivalence; this is one of the basic homotopy invariance properties of fat realization, and is stated in this form by Ebert and Randal-Williams \cite[Theorem 2.2]{EbertRandalWilliamsSemiSimplicialSpaces}. This result is \emph{false} for the geometric realization as shown by this nice, simple example given by Tyler Lawson on MathOverflow \cite{Lawson}:

\begin{example}
    Let $X$ be the simplicial space with $X_0 = \{0\}$, $X_1 = \{0\} \cup \{1/n : n > 0\}$ and no higher non-degenerate simplices (which means that $X_n = X_1 \vee_{X_0} X_1 \vee_{X_0} \cdots \vee_{X_0} X_1$). Then the geometric realization of $X$ is homeomorphic to the Hawaiian earring. Now consider the levelwise discretization of $X$, say $X^\delta$. The identity map $X^\delta \to X$ is a levelwise weak equivalence, but $|X^\delta|$ is a countable wedge of circles, which is not weakly equivalent to the Hawaiian earring, since they famously do no have isomorphic fundamental groups.
\end{example}

Second, for a simplicial space $X_\bullet$, there is a natural comparison map from the fat realization to the ordinary geometric realization, given by further quotienting out by the action of the degeneracies:
\[
    \|X_\bullet\|\longrightarrow |X_\bullet|.
\]
This map is a homotopy equivalence when $X_\bullet$ is \emph{good}, meaning that each degeneracy map $s_i\colon X_n\to X_{n+1}$ is a cofibration. This comparison theorem is due to Segal \cite[Theorem A.1(iv)]{Segal}. When $G$ is either discrete or a Lie group, the simplicial spaces used to define $E\com G$ and $B\com G$ are good, so in that context both types of realization give homotopy equivalent answers. For our group $\SL_2(\Z[1/p])$ there is a marked difference in behavior. Because fat realization preserves weak equivalences, we believe this is a better choice to define $E\com G$ and $B\com G$ for general topological groups and those notations from now on refer to the fat realization version of the definitions. When we want to use the geometric realization we will use the notations $E^{geo}\com G$ and $B^{geo}\com G$.

\begin{observation}\label[observation]{obs:z-localized-totally-disconnected}
The space \(\Z[1/p]\), with the topology induced from \(\R\), is totally disconnected.
\end{observation}

We will use the following obvious topological lemma:

\begin{lemma}\label[lemma]{lem:totally-disconnected-weak-equivalence}
Let \(X\) be a totally disconnected topological space. Then the identity homomorphism
\[
    X^\delta\longrightarrow X
\]
is a weak equivalence.
\end{lemma}

\begin{proof}
Every continuous map from a connected space into a totally disconnected space is constant. Hence the path components of \(X\) are points, just as for \(X^\delta\), and all higher homotopy groups vanish. Therefore the identity map induces a bijection on path components and isomorphisms on all homotopy groups.
\end{proof}

\begin{proposition}\label[proposition]{prop:ecom-bcom-discrete-topological-weak-equivalence}
If \(G\) is totally disconnected, then we have weak equivalences
\[
    E\com G^\delta\simeq E\com G
    \quad \text{and} \quad
    B\com G^\delta\simeq B\com G.
\]
In particular,
\[
    E\com\SL_2(\Z[1/p])^\delta\simeq
    E\com\SL_2(\Z[1/p])
\]
and
\[
    B\com\SL_2(\Z[1/p])^\delta\simeq
    B\com\SL_2(\Z[1/p]).
\]
\end{proposition}

\begin{proof}
The identity map \(G^\delta\to G\) induces weak equivalences degreewise on the simplicial spaces defining \(B\com\) and \(E\com\), since in both cases the spaces of $n$-simplices are totally disconnected for each $n$. Thus the induced maps on fat realizations are weak equivalences by the homotopy-invariance of fat realization recalled above.
\end{proof}

We will now show that $B\com \SL_2(\Z[1/p])$ and $B\com \SL_2(\Z[1/p])^\delta$ are not homotopy equivalent. For this we shall turn to the quasitopological fundamental group. For a based space \(X\), let
$\pi_1^{\mathrm{qtop}}(X)$ denote the quotient \(\Omega X/{\sim}\), where \(\sim\) is the relation of based homotopy, endowed with the quotient topology. Perhaps surprisingly, this topology does not make the fundamental group into a topological group, but it is at least a quasitopological group, which means that the inverse map is continuous and that the multiplication map is continuous in each variable separately \cite{Brazas}. Paul Fabel has shown that for the Hawaiian earring $H$, $\pi_1^{\mathrm{qtop}}(H)$ is not a topological group \cite{Fabel}. 

\begin{lemma}\label[lemma]{lem:qtop-bg-discrete}
If \(G\) is discrete, then \(\pi_1^{\mathrm{qtop}}(BG)\) is discrete.
\end{lemma}

\begin{proof}
When $G$ is discrete, the simplicial model of $BG$ is a CW complex, and for those $\pi_1^{\mathrm{qtop}}(X)$ is always discrete. In fact, more generally, when $X$ is locally path connected
and semilocally simply connected, then $\pi_1^{\mathrm{qtop}}(X)$ is discrete \cite[Theorem 2.12]{Brazas}.
\end{proof}

\begin{corollary}\label[corollary]{cor:qtop-bcom-discrete-version-discrete}
The group
\[
    \pi_1^{\mathrm{qtop}}\bigl(B\com\SL_2(\Z[1/p])^\delta\bigr)
\]
is discrete.
\end{corollary}

\begin{proof}
This follows directly from Theorem \ref{thm:bcom-sl2-localized-amalgam} and Lemma \ref{lem:qtop-bg-discrete}.
\end{proof}

\begin{proposition}\label[proposition]{prop:qtop-bcom-topological-version-not-discrete}
The group
\[
    \pi_1^{\mathrm{qtop}}\bigl(B\com\SL_2(\Z[1/p])\bigr)
\]
is not discrete.
\end{proposition}

\begin{proof}
This follows from Proposition \ref{prop:qtop-discrete-implies-group-discrete}, since \(\SL_2(\Z[1/p])\) with the topology induced from \(\SL_2(\R)\) is totally disconnected but not discrete.
\end{proof}

We will use the following simple result: if \(X\) and \(Y\) are homotopy equivalent, then \(\pi_1^{\mathrm{qtop}}(X)\) and \(\pi_1^{\mathrm{qtop}}(Y)\) are homeomorphic. This is simply because continuous maps induce continuous maps between the topologized fundamental groups, and homotopic maps induce equal maps between the fundamental groups. See the discussion in the third paragraph of the introduction of \cite{BrazasFabel}.

\begin{proposition}\label[proposition]{prop:qtop-discrete-implies-group-discrete}
Let \(G\) be a totally disconnected topological group. If
\[
    \pi_1^{\mathrm{qtop}}(B\com G)
\]
is discrete, then \(G\) is discrete.
\end{proposition}

\begin{proof}
The \(1\)-skeleton of \(B\com G\) is \(\Sigma G\). Let
\[
    i\colon G\longrightarrow \Omega B\com G
\]
be the adjoint of the inclusion \(\Sigma G\to B\com G\). Thus \(i(g)\) is the loop represented by the edge labelled by \(g\). Let
\[
    k\colon \Omega B\com G\longrightarrow
    \Omega BG
\]
be the map induced by the canonical map \(B\com G\to BG\) (where we also use fat realization for $BG$), and let
\[
    p\colon \Omega B\com G\longrightarrow
    \Omega B\com G/{\sim}
    =\pi_1^{\mathrm{qtop}}(B\com G)
\]
and
\[
    f\colon \Omega BG\longrightarrow
    \Omega BG/{\sim}
    =\pi_1^{\mathrm{qtop}}(BG)
\]
be the quotient maps by the relation of based homotopy. With this notation we have the commutative diagram
\[
\xymatrix{
    G \ar[r]^-{i} \ar[dr] & \Omega B\com G \ar[r]^-{p} \ar[d]_-{k} & (\Omega B\com G)/{\sim} \ar[d]^-{\ell} \\
    & \Omega BG \ar[r]_-{f} & (\Omega BG)/{\sim}.
}
\]
Here \(\ell\) is induced by \(k\), so that \(\ell p=fk\). Since \(G\) is totally disconnected, the identity map
\[
    G^\delta\longrightarrow G
\]
is a weak equivalence by \cref{lem:totally-disconnected-weak-equivalence}; and, again by the homotopy-invariance of fat realization, we get an induced weak homotopy equivalence $BG^\delta \to BG$. In particular, the fundamental group of $BG$ is isomorphic to $G$ as a group, just as in the $G^\delta$ case. Thus, the composite $fki : G \to (\Omega BG)/{\sim}$ is a bijection. Since $fki = lpi$, this implies $pi$ is injective and therefore
\[
    (p i)^{-1}\bigl(p(i(g))\bigr)=\{g\}.
\]
If \(\pi_1^{\mathrm{qtop}}(B\com G)\) is discrete, then the singleton \(\{p(i(g))\}\) is open, and consequently its inverse image \(\{g\}\) under the continuous map \(pi\) is open in \(G\). Thus every singleton of \(G\) is open, and \(G\) is discrete.
\end{proof}

\begin{corollary}\label[corollary]{cor:bcom-discrete-nondiscrete-weak-not-homotopy-equivalent}
Let \(G=\SL_2(\Z[1/p])\), endowed with the topology induced from \(\SL_2(\R)\). Then the spaces
\[
    B\com G^\delta
    \quad \text{and} \quad
    B\com G
\]
are weakly homotopy equivalent, but they are not homotopy equivalent.
\end{corollary}

\begin{proof}
By \cref{obs:z-localized-totally-disconnected}, the group \(G\) is totally disconnected. Therefore \cref{prop:ecom-bcom-discrete-topological-weak-equivalence} implies that the natural comparison map
\[
    B\com G^\delta\longrightarrow B\com G
\]
is a weak homotopy equivalence.

On the other hand, \cref{cor:qtop-bcom-discrete-version-discrete} shows that
\[
    \pi_1^{\mathrm{qtop}}(B\com G^\delta)
\]
is discrete, while \cref{prop:qtop-bcom-topological-version-not-discrete} shows that
\[
    \pi_1^{\mathrm{qtop}}(B\com G)
\]
is not discrete. If \(B\com G^\delta\) and \(B\com G\) were homotopy equivalent, then the theorem of Brazas--Fabel recalled above would imply that these two quasi-topological fundamental groups are homeomorphic. This is impossible, since discreteness is a topological invariant. Hence the two spaces are not homotopy equivalent.
\end{proof}

\begin{remark}
    For the version of these spaces constructed using the geometric realization rather than the fat realization, it seems very likely that $B^{geo}\com \SL_2(\Z[1/p])$ and $B\com \SL_2(\Z[1/p])^\delta$ are not even weakly equivalent, and indeed, Lawson's example above leads us to think that $\pi_1(B^{geo}\com \SL_2(\Z[1/p]))$ most likely contains copies of the fundamental group of the Hawaiian earring, and thus it is not even countable.
\end{remark}

\begin{question} What about topological groups with more intricate topologies?
It would be interesting to carry out this kind of comparison between $B\com G^\delta$ and $B\com G$ (using, as we recommend, the fat realization to define them) for topological groups with a significant topology that is somewhere between a Lie group topology and ``dust'', that is, totally disconnected.
\end{question}

\bibliographystyle{alpha}
\bibliography{myblib}

@article{AGS,
  title={The classifying space for commutativity of geometric orientable 3-manifold groups},
  author={Antol{\'\i}n-Camarena, Omar and Garc{\'\i}a-Hern{\'a}ndez, Luis Eduardo and Salda{\~n}a, Luis Jorge S{\'a}nchez},
  journal={Topology and its Applications},
  volume={373},
  pages={109484},
  year={2025},
  publisher={Elsevier}
}

@Article{AdemCohen2007,
  Title                    = {Commuting elements and spaces of homomorphisms},
  Author                   = {Adem, Alejandro and Cohen, Frederick R.},
  Journal                  = {Math. Ann.},
  Year                     = {2007},
  Number                   = {3},
  Pages                    = {587--626},
  Volume                   = {338},

  Doi                      = {10.1007/s00208-007-0083-4},
  Fjournal                 = {Mathematische Annalen}
}

@Article{AdemGomez2015,
  Title                    = {A classifying space for commutativity in {L}ie groups},
  Author                   = {Adem, Alejandro and G\'omez, Jos\'e Manuel},
  Journal                  = {Algebr. Geom. Topol.},
  Year                     = {2015},
  Number                   = {1},
  Pages                    = {493--535},
  Volume                   = {15},

  Doi                      = {10.2140/agt.2015.15.493},
  Fjournal                 = {Algebraic \& Geometric Topology}
}

@Book{Ha02,
  Title                    = {Algebraic topology},
  Author                   = {Hatcher, Allen},
  Publisher                = {Cambridge University Press, Cambridge},
  Year                     = {2002},

  ISBN                     = {0-521-79160-X; 0-521-79540-0},
  Mrclass                  = {55-01 (55-00)},
  Mrnumber                 = {1867354 (2002k:55001)},
  Mrreviewer               = {Donald W. Kahn},
  Pages                    = {xii+544}
}

@article{Segal,
title = {Categories and cohomology theories},
journal = {Topology},
volume = {13},
number = {3},
pages = {293-312},
year = {1974},
issn = {0040-9383},
doi = {https://doi.org/10.1016/0040-9383(74)90022-6},
url = {https://www.sciencedirect.com/science/article/pii/0040938374900226},
author = {Graeme Segal}
}

@misc{Lawson,
    TITLE = {Is the geometric realization of a level-wise weak equivalence a weak equivalence?},
    AUTHOR = {Tyler Lawson (https://mathoverflow.net/users/360/tyler-lawson)},
    HOWPUBLISHED = {MathOverflow},
    NOTE = {URL:https://mathoverflow.net/q/171423 (version: 2014-06-09)},
    EPRINT = {https://mathoverflow.net/q/171423},
    URL = {https://mathoverflow.net/q/171423}
}

@article{Brazas,
title = {The topological fundamental group and free topological groups},
journal = {Topology and its Applications},
volume = {158},
number = {6},
pages = {779-802},
year = {2011},
issn = {0166-8641},
doi = {https://doi.org/10.1016/j.topol.2011.01.022},
url = {https://www.sciencedirect.com/science/article/pii/S0166864111000368},
author = {Jeremy Brazas},
}

@article{Fabel,
author = {Paul Fabel},
journal = {Bulletin of the Polish Academy of Sciences. Mathematics},
language = {eng},
number = {1},
pages = {77-83},
title = {Multiplication is Discontinuous in the Hawaiian Earring Group (with the Quotient Topology)},
url = {http://eudml.org/doc/281130},
volume = {59},
year = {2011},
}

@article{BrazasFabel,
  title={On fundamental groups with the quotient topology},
  author={Brazas, Jeremy and Fabel, Paul},
  journal={Journal of Homotopy and Related Structures},
  volume={10},
  number={1},
  pages={71--91},
  year={2015},
  publisher={Springer}
}

@article{AdemCohenTorresGiese,
  author  = {Adem, Alejandro and Cohen, Frederick R. and Torres Giese, Enrique},
  title   = {Commuting elements, simplicial spaces and filtrations of classifying spaces},
  journal = {Mathematical Proceedings of the Cambridge Philosophical Society},
  volume  = {152},
  number  = {1},
  pages   = {91--114},
  year    = {2012},
  doi     = {10.1017/S0305004111000739}
}

@article{AntolinGritschacherVillarreal2023,
  title={Higher generation by abelian subgroups in Lie groups},
  author={Antol{\'\i}n-Camarena, Omar and Gritschacher, Simon and Villarreal, Bernardo},
  journal={Transformation Groups},
  volume={28},
  number={4},
  pages={1375--1390},
  year={2023},
  publisher={Springer}
}

@article{EbertRandalWilliamsSemiSimplicialSpaces,
  author  = {Ebert, Johannes and Randal-Williams, Oscar},
  title   = {Semi-simplicial spaces},
  journal = {Algebraic {\&} Geometric Topology},
  volume  = {19},
  number  = {4},
  pages   = {2099--2150},
  year    = {2019},
  doi     = {10.2140/agt.2019.19.2099}
}

@book{SerreTrees,
  author    = {Serre, Jean-Pierre},
  title     = {Trees},
  translator = {Stillwell, John},
  publisher = {Springer-Verlag},
  address   = {Berlin},
  year      = {1980},
  series    = {Springer Monographs in Mathematics}
}

@book{NeukirchANT,
  author    = {Neukirch, J\"urgen},
  title     = {Algebraic Number Theory},
  series    = {Grundlehren der mathematischen Wissenschaften},
  volume    = {322},
  publisher = {Springer-Verlag},
  address   = {Berlin},
  year      = {1999}
}
\end{document}